\documentclass{amsart}

\usepackage{amsmath}
\usepackage{amsthm}
\usepackage{amsfonts}
\usepackage{amssymb}
\usepackage[all]{xy}
\usepackage{url}

\DeclareMathSymbol{\rightrightarrows}  {\mathrel}{AMSa}{"13}

\def\Ob{\operatorname{Ob}}

\def\Mor{\operatorname{Mor}}
\def\St{\operatorname{St}}

\def\Tot{\operatorname{Tot}}

\def\sk{\operatorname{sk}}

\def\Ext{\operatorname{Ext}}

\def\pb{\operatorname{pb}}

\def\Tr{\operatorname{Tr}}

\catcode`\@=11
\def\varholim@#1#2{\mathop{\vtop{\ialign{##\crcr
 \hfil$#1\m@th\operator@font holim$\hfil\crcr
 \noalign{\nointerlineskip\kern\ex@}#2#1\crcr
 \noalign{\nointerlineskip\kern-\ex@}\crcr}}}}
\def\hocolim{\mathpalette\varholim@\rightarrowfill@} 
\def\hoinvlim{\mathpalette\varholim@\leftarrowfill@}
\catcode`\@=\active

\newtheorem{theorem}{Theorem}
\newtheorem{lemma}[theorem]{Lemma}
\newtheorem{corollary}[theorem]{Corollary}

\theoremstyle{definition}

\newtheorem{example}[theorem]{Example}

\newtheorem{remark}[theorem]{Remark}

\begin{document}

\title{Cosimplicial spaces and cocycles}
                                                                                \author{J.F. Jardine}
\thanks{This research was supported by the Natural Sciences and Engineering Research Council of Canada}
\address{Department of Mathematics\\University of Western Ontario\\
London, Ontario N6A 5B7 \\
Canada}
 
\email{jardine@uwo.ca}
 



\begin{abstract}
\noindent
Standard results from non-abelian cohomology theory specialize to a theory of torsors and stacks for cosimplicial groupoids. The space of global sections of the stack completion of a cosimplicial groupoid $G$ is weakly equivalent to the Bousfield-Kan total complex of $BG$ for all cosimplicial groupoids $G$. The $k$-invariants for the Postnikov tower of a cosimplicial space $X$ are naturally elements of stack cohomology for the stack associated to the fundamental groupoid $\pi(X)$ of $X$. 
Cocycle-theoretic ideas and techniques are used throughout the paper.
\end{abstract}

\maketitle

\section*{Introduction}

This paper is an exposition of the basic homotopy theory of cosimplicial
spaces, from a point of view that is informed by sheaf
theoretic homotopy theory.

This discussion interpolates ideas associated with the
injective model structure for cosimplicial spaces with classical
results of Bousfield and Kan. The injective model structure for
cosimplicial spaces is a special case of the injective model structure
for all small diagrams of simplicial sets $I \to s\mathbf{Set}$ which
are defined on a fixed index category $I$, and this in turn is a
special case of the injective model structure for simplicial
sheaves (and presheaves) on a small Grothendieck site.

We effectively lose nothing by working within the injective structure
for cosimplicial spaces, as it has the same weak equivalences as the
Bousfield-Kan model structure.  At the same time, interesting
phenomena arise from the injective structure which correspond to well
studied features of the homotopy theory of simplicial sheaves.

In particular, the injective model structure 
creates an attractive theory of torsors and stacks for cosimplicial
groupoids, which is displayed in the second section of this paper. As
in local homotopy theory, the category of cosimplicial groupoids has a
model structure which is induced from the injective structure on
cosimplicial spaces, for which the fibrant object associated to a
cosimplicial groupoid $H$ is its stack completion $\St(H)$. The stack
completion may be described in global sections by torsors, suitably
defined, and the link between torsors and stacks is achieved by using
cocycles. The use of cocycle theory is a recurring theme of
this paper.

There is a minor surprise: while the cosimplicial space $BG$
associated to a cosimplicial groupoid $G$ might not be Bousfield-Kan
fibrant, any weak equivalence $G \to H$ induces a weak equivalence of
the associated Bousfield-Kan total complexes. This is a consequence of
Theorem \ref{th 12} (or Corollary \ref{cor 13}) below, which says that
the total complex $\Tot(BG)$ and the classifying space
$B(G-\mathbf{tors})$ of the groupoid of $G$-torsors
have the same homotopy type.

The set of isomorphism classes of $G$-torsors, or equivalently the set
of path components of the groupoid $G-\mathbf{tors}$, coincides with
the set of morphisms $[\ast,BG]$ in the homotopy category of
cosimplicial spaces, just as in sheaf theory.

The rest of the paper (especially Section 4) is taken up with an
analysis of Postnikov towers and $k$-invariants. 

The Postnikov tower of a cosimplicial space $X$ is used to construct
an analog of the cohomological descent spectral sequence for the
homotopy inverse limit of $X$, as one would expect \cite{GECT}, modulo the catch
that $X$ has to have a non-trivial cocycle for this approach to say
anything at all. In fact, one can prove easily (Lemma \ref{lem 6})
that $X$ has a non-empty homotopy inverse limit if and only if there
is a cocycle
\begin{equation*}
\ast \xleftarrow{\simeq} U \to X.
\end{equation*}
There are injective fibrant cosimplicial spaces which do not have cocycles 
--- see Example \ref{ex 5}.

To analyze the Postnikov tower of a cosimplicial space $X$ away from
the cocycles of $X$, a different method is required.

The $k$-invariant of the standard fibration $P_{n}Y \to P_{n-1}Y$ for
a simply connected Kan complex $Y$ can be described as the composite
\begin{equation*}
P_{n-1}Y \to P_{n-1}Y/P_{n}Y \to P_{n+1}(P_{n-1}Y/P_{n}Y) =: Z_{n}Y,
\end{equation*}
for $n \geq 2$, and $Z_{n}Y$ has a functorial base point. 
It follows
(Lemma \ref{lem 23}) that, for any diagram of simply connected Kan
complexes $X$, there is a weak equivalence of diagrams
\begin{equation*}
Z_{n}X \simeq K(H_{n}(Z_{n}X),n+1).
\end{equation*} 
The resulting fibre homotopy sequence
\begin{equation*}
P_{n}X \to P_{n-1}X \to Z_{n}X
\end{equation*}
specializes to a fibre sequence of diagrams which is indexed by the
stack completion of the fundamental groupoid
$\pi(X)$. The homotopy cartesian square which is given by Theorem
\ref{th 26} is the result of applying a homotopy colimit functor to
that sequence. From this point of view, the $k$-invariants of $X$
are elements of stack cohomology groups that are associated to the fundamental
groupoid $\pi(X)$.

There is nothing special about cosimplicial spaces in the
$k$-invariant construction --- that same construction applies to all
diagram-theoretic homotopy types.
\medskip

I would like to thank the referee for a series of helpful
comments, which led to a substantial sharpening of the
exposition of this paper.

\vfill\eject

\tableofcontents

\section{The injective model structure}

Suppose that $\mathbf{\Delta}$ is the category of finite ordinal
numbers $\mathbf{n} = \{ 0, 1, \dots ,n\}$, $n \geq 0$, and all
order-preserving functions between them.

Write $s\mathbf{Set}^{\mathbf{\Delta}}$ for the category of
cosimplicial spaces, meaning functors 
\begin{equation*}
X: \mathbf{\Delta} \to
s\mathbf{Set} 
\end{equation*}
taking values in simplicial sets, and their natural
transformations. It is standard practice to write $X^{n} =
X(\mathbf{n})$ for $\mathbf{n} \in \mathbf{\Delta}$.

The injective model structure on the category of cosimplicial spaces
has weak equivalences and cofibrations defined {\it
  sectionwise}\footnote{The term ``sectionwise'' is commonly used in
  Algebraic Geometry. Homotopy theorists more often use 
  ``objectwise'' or ``pointwise'' to describe the same concept.}: a
map $X \to Y$ is a {\it weak equivalence} (respectively {\it
  cofibration}) if and only if all maps $X^{n} \to Y^{n}$ are weak
equivalences (respectively cofibrations) of simplicial sets. The {\it
  injective fibrations} are defined by a right lifting property with
respect to trivial cofibrations.

The weak equivalences coincide with the weak equivalences of the
Bousfield-Kan model structure on cosimplicial spaces \cite{BK}. A {\it
  Bousfield-Kan cofibration} is a sectionwise cofibration which
induces an isomorphism in maximal augmentations. Explicitly, the
maximal augmentation $X^{-1}$ of a cosimplicial set is the simplicial
set which is defined by the equalizer diagram
\begin{equation*}
\xymatrix@C=12pt{
X^{-1} \ar[r] & X^{0} \ar@<1ex>[r]^{d^{0}} \ar@<-1ex>[r]_{d^{1}} & X^{1}
}
\end{equation*}
Thus, a Bousfield-Kan cofibration is a cofibration $A \to B$ as
defined above, such that the map $A^{-1} \to B^{-1}$ is an
isomorphism. 

It follows that every injective fibration is
a Bousfield-Kan fibration.

We also have the following:

\begin{lemma}\label{lem 1}
There is a natural isomorphism
\begin{equation*}
\varprojlim_{n}\ X^{n} \xrightarrow{\cong} X^{-1}
\end{equation*}
for cosimplicial sets (hence for cosimplicial spaces) $X$.
\end{lemma}

\noindent
The proof of Lemma \ref{lem 1} is elementary.

The simplicial set $\Tot(Y)$ is usually defined \cite{BK} for a Bousfield-Kan
fibrant cosimplicial space $Y$ by
\begin{equation*}
\Tot(Y) = \mathbf{hom}(\Delta,Y),
\end{equation*}
where $\Delta$ is the cosimplicial space $\mathbf{n} \mapsto
\Delta^{n}$ and $\mathbf{hom}(\Delta,Y)$ is the usual diagram
theoretic function complex.

In general, for cosimplicial spaces $X$ and $Y$, the function complex
$\mathbf{hom}(X,Y)$ is the simplicial set whose $n$-simplices are the
cosimplicial space maps 
\begin{equation*}
X \times \Delta^{n} \to Y.
\end{equation*}
This function complex construction defines a closed simplicial model
structure for both the injective and Bousfield-Kan model structure on
cosimplicial spaces.

The notation $\ast$ is used for the terminal object in cosimplicial
spaces --- it is the constant diagram on the one-point simplicial
set. The cosimplicial space $\Delta$ is cofibrant for the
Bousfield-Kan model structure, and it is a ``fat point'' in the sense
that the canonical map $\Delta \to \ast$ is a weak equivalence.
\medskip

It is now standard to say (see \cite{J40}, for example) that the
homotopy inverse limit $\hoinvlim_{\mathbf{n}}\ X^{n}$ of a
cosimplicial space $X$ is defined by taking an injective fibrant model
$j: X \to Z$ (a weak equivalence with $Z$ injective fibrant), and then
setting
\begin{equation*}
\hoinvlim_{\mathbf{n}}\ X^{n} = \varprojlim{}_{\mathbf{n}} Z^{n} = 
\mathbf{hom}(\ast,Z).
\end{equation*}
In this sense, the homotopy inverse limit is a derived inverse limit.

The injective  model structure  on cosimplicial spaces  is cofibrantly
generated, so that one can make  a natural choice of injective fibrant
model.  The  homotopy inverse  limit  construction  just described  is
therefore functorial in cosimplicial spaces $X$.

\begin{lemma}
There is a natural weak equivalence
\begin{equation*}
\Tot(Y) \simeq \hoinvlim_{\mathbf{n}}\ Y^{n}
\end{equation*}
for Bousfield-Kan fibrant objects $Y$.
\end{lemma}

\begin{proof}
Every injective fibrant cosimplicial space is
Bousfield-Kan fibrant.

The cosimplicial space $\Delta$ is cofibrant for the
Bousfield-Kan structure, so that any weak equivalence $Y \to Y'$ of
Bousfield-Kan fibrant objects induces a weak equivalence
\begin{equation*}
\mathbf{hom}(\Delta,Y) \to \mathbf{hom}(\Delta,Y').
\end{equation*}
Thus, if $j: Y \to Z$ is an injective fibrant model for a
Bousfield-Kan fibrant object $Y$, then the map
\begin{equation*}
\mathbf{hom}(\Delta,Y) \xrightarrow{j_{\ast}} \mathbf{hom}(\Delta,Z)
\end{equation*}
is a weak equivalence. At the same time, the map $\Delta \to \ast$ is
a weak equivalence of cofibrant objects for the injective model
structure on cosimplicial spaces, so that the induced map
\begin{equation*}
\mathbf{hom}(\ast,Z) \to \mathbf{hom}(\Delta,Z)
\end{equation*}
is a weak equivalence since $Z$ is injective fibrant. 
\end{proof}

I shall now write 
\begin{equation*}
\Tot(X) = \hoinvlim_{\mathbf{n}}\ X^{n}
\end{equation*} 
for all cosimplicial spaces $X$.

There are natural identifications
\begin{equation}
\pi_{0}\Tot(X) = [\ast,X]
\end{equation}
and
\begin{equation}\label{eq 2}
\pi_{n}(\Tot(X),x) \cong [S^{n},X]_{\ast}
\end{equation}
with morphisms in the homotopy category (respectively, pointed
homotopy category) of cosimplicial spaces, where $x$ is a cosimplicial
space map $x: \ast \to X$, or a {\it global} base point for $X$. The
pointed simplicial set $S^{n} = (S^{1})^{\wedge n}$ is
identified with a constant cosimplicial space in the formula (\ref{eq 2}).

Here is another elementary statement:

\begin{lemma}\label{lem 3}
Suppose that the cosimplicial space $X$ is a cosimplicial set in the
sense that the simplicial set $X^{n}$ is discrete on a set of vertices for
all $n$. Then $X$ is injective fibrant.
\end{lemma}

\begin{proof}
If the map $i: A \to B$ is a trivial cofibration of cosimplicial sets
then the induced map $\pi_{0}A \to \pi_{0}B$ is an isomorphism of
cosimplicial sets, and any map $A \to X$ factors uniquely through a
map $\pi_{0}A \to X$. Thus, all lifting problems
\begin{equation*}
\xymatrix{
A \ar[r] \ar[d]_{i} & \pi_{0}A \ar[r] \ar[d]^{\cong} & X \\
B \ar[r] & \pi_{0}B \ar@{.>}[ur]
}
\end{equation*}
can be solved.
\end{proof}

\begin{remark}\label{rem 4}
Lemma \ref{lem 3} is a special case of a basic sheaf theoretic fact. If $F$
is a sheaf of sets on a small Grothendieck site $\mathcal{C}$, then
the simplicial sheaf $K(F,0)$ is fibrant for the injective model
structure for simplicial presheaves on $\mathcal{C}$ which is defined
by the topology --- this statement appears, for example, as Lemma 6.10 of
\cite{GECT}, with the same proof.  

Suppose that $I$ is a small category. In the $I$-diagram category
$s\mathbf{Set}^{I}$ in simplicial sets, every presheaf is a sheaf, and
so every $I$-diagram of simplicial sets which is simplicially constant
is injective fibrant.
\end{remark}

\begin{example}\label{ex 5}
There are cosimplicial spaces $X$ for which $\Tot(X) = \emptyset$.
The cosimplicial space $\Delta$ has empty inverse limit, and it follows
that the cosimplicial space $\sk_{0}\Delta$ (vertices of $\Delta^{n}$ for
all $n$) has empty inverse limit. This object is an injective fibrant
cosimplicial space, by Lemma \ref{lem 3}.
\end{example}

A {\it cocycle} $(g,f)$ from $X$ to $Y$ is a diagram in cosimplicial spaces
\begin{equation*}
X \underset{\simeq}{\xleftarrow{g}} V \xrightarrow{f} Y,
\end{equation*}
such that $g$ is a weak equivalence.
A morphism of cocycles is a commutative diagram
\begin{equation*}
\xymatrix@R=8pt{
& V \ar[dl]_{\simeq} \ar[dd] \ar[dr] \\
X && Y \\
& V' \ar[ul]^{\simeq} \ar[ur]
}
\end{equation*}
These are the objects and morphisms of the cocycle category
$h(X,Y)$. It is a basic result \cite{coc-cat} for injective model
structures on diagram categories that the assignment which sends a
cocycle $(g,f)$ to the morphism $fg^{-1}$ in the
homotopy category defines a bijection
\begin{equation*}
\pi_{0}Bh(X,Y) \xrightarrow{\cong} [X,Y]
\end{equation*} 
between path components of the cocycle category $h(X,Y)$ and the set
of morphisms $[X,Y]$ of the homotopy category.

We also have the following:

\begin{lemma}\label{lem 6}
Suppose that $X$ is a cosimplicial space. Then the space $\Tot(X)$ is
non-empty if and only if there is a cocycle
\begin{equation*}
\ast \xleftarrow{\simeq} U \to X.
\end{equation*}
\end{lemma}

\begin{proof}
The space $\Tot(X)$ is non-empty if and only if an injective fibrant
model $j: X \to Z$ for $X$ has a vertex $\ast \to Z$. We show that the
object $Z$ has a global vertex $\ast \to Z$ if and only if the cocycle
category $h(\ast,Z)$ is non-empty. The map $j$ is a weak
equivalence, so that the cocycle category $h(\ast,Z)$ is non-empty if
and only if $h(\ast,X)$ is non-empty, since these two categories have
isomorphic sets of path components.

To see that the injective fibrant object $Z$ has a map $\ast \to Z$ if
the cocycle category $h(\ast,Z)$ is non-empty, let 
\begin{equation*}
\ast
\xleftarrow{\simeq} U \xrightarrow{f} Z
\end{equation*} 
be a cocycle, and observe
that there is a commutative diagram
\begin{equation*}
\xymatrix@R=12pt{
U \ar[r]^{f} \ar[d]_{j} & Z \\
U' \ar[ur]_{f'}
}
\end{equation*}
where $j$ is a trivial cofibration and $U'$ is injective fibrant.
The map $U' \to \ast$ is a trivial injective fibration, and
therefore has a section $\ast \to U'$, and there is a map
\begin{equation*}
\ast \to U' \xrightarrow{f'} Z.
\end{equation*} 
\end{proof}

\begin{lemma}
Suppose that the map $p: X \to Y$ is an injective fibration between injective
fibrant cosimplicial spaces. Suppose that $\Tot(Y) \ne \emptyset$ and let
$x \in \Tot(Y)$ be a vertex. Let $F$ be the fibre over $x$, so that
the diagram
\begin{equation*}
\xymatrix{
F \ar[r]^{i} \ar[d] & X \ar[d]^{p} \\
\ast \ar[r]_{x} & Y
}
\end{equation*}
is a pullback in cosimplicial spaces. Then $\Tot(F) \ne \emptyset$ if
and only if there is a cocycle $\ast \xleftarrow{\simeq} U
\xrightarrow{f} X$ such that the composite cocycle
\begin{equation*}
\ast \xleftarrow{\simeq} U \xrightarrow{pf} Y
\end{equation*}
is in the path component of the cocycle $x: \ast \to Y$.
\end{lemma}

\begin{proof}
If $\Tot(F) \ne \emptyset$ then there is a cocycle 
\begin{equation*}
\ast \xleftarrow{\simeq} U \xrightarrow{g} F, 
\end{equation*}
and then the diagram
\begin{equation*}
\xymatrix{
U \ar[r]^{ig} \ar[d] & X \ar[d]^{p} \\
\ast \ar[r]_{x} & Y
}
\end{equation*}
commutes.

Conversely, suppose given a map of cocycles
\begin{equation*}
\xymatrix@R=10pt{
U_{1} \ar[dr]^{f_{1}} \ar[dd] \\
& Y \\
U_{2} \ar[ur]_{f_{2}}
}
\end{equation*}
and write $F_{i} = U_{i} \times_{Y} X$ for $i=1,2$. If either of the
maps $f_{1}$ or $f_{2}$ lifts to $X$ then some $F_{i}$ has a
non-trivial cocycle. Both of the objects $F_{i}$ have non-trivial
cocycles since the map $F_{1} \to F_{2}$ is a weak equivalence. It
follows that if the cocycle $x: \ast \to Y$ is in the path
component of a cocycle $U \to Y$ that lifts to $X$, then the fibre $F$
has a non-trivial cocycle.
\end{proof}

We finish this section by recalling some basic notation and
concepts from \cite{BK}, which will be needed later.

Write $\mathbf{\Delta}_{\leq n}$ for the full subcategory of
$\mathbf{\Delta}$ on the ordinal numbers $\mathbf{k}$ with $k\leq
n$. Write $Tr_{n}X$ for the composite functor
\begin{equation*}
\mathbf{\Delta}_{\leq n} \subset \mathbf{\Delta} \xrightarrow{X} s\mathbf{Set},
\end{equation*}
and let $L_{n}$ be the left adjoint of the truncation functor $X
\mapsto Tr_{n}X$. 
The $n$-skeleton $\sk_{n}Y$ of a simplicial set $Y$ can be defined by
\begin{equation*}
\sk_{n}Y = \varinjlim_{\Delta^{k} \to Y, k \leq n} \ \Delta^{k}.
\end{equation*}
It follows that there is an isomorphism of cosimplicial spaces
\begin{equation*}
\sk_{n}\Delta \cong L_{n}Tr_{n}\Delta.
\end{equation*}

This relationship between skeleta and truncations is used to show
that a map $f: \sk_{n-1}\Delta \to X$ can be extended to a map $f':
\sk_{n}\Delta \to X$ if and only if there is a map (simplex) $f':
\Delta^{n} \to X^{n}$ such that the diagram
\begin{equation*}
\xymatrix{
\sk_{n-1}\Delta^{n} \ar[r]^-{i} \ar[dr]_{f} 
& \Delta^{n} \ar[d]^{f'} \ar[r]^-{s}
& M^{n-1}\Delta \ar[d]^{f_{\ast}} \\
& X^{n} \ar[r]_-{s} & M^{n-1}X
}
\end{equation*}
commutes. Here, the {\it matching space} $M^{n-1}X$ is defined by the
assignment
\begin{equation*}
M^{n-1}X = \varprojlim_{\mathbf{n} \to \mathbf{k}, k< n}\ X^{k} \cong 
\varprojlim_{\mathbf{n} \overset{s}{\twoheadrightarrow} \mathbf{k}, k<n}\ X^{k},
\end{equation*}
which inverse limit can also be defined by the equalizer
\begin{equation*}
\xymatrix{
M^{n-1}X \ar[r] & \prod_{0 \leq i \leq n-1}\ X^{n-1} \ar@<1ex>[r] \ar@<-1ex>[r] 
& \prod_{0 \leq i \leq j \leq n-1}\ X^{n-2}
}
\end{equation*}
which arises from the cosimplicial identities $s^{j}s^{j} =
s^{j+1}s^{i}$, $i \leq j$. The canonical map $s: X^{n} \to M^{n-1}X$
is induced by the map
\begin{equation*}
(s^{i}): X^{n} \to \prod_{0 \leq i \leq n-1}\ X^{n-1}
\end{equation*}
that is defined by the codegeneracies $s^{i}$.

\section{Torsors}

Suppose that $H$ is a cosimplicial groupoid, with source and target
maps $s,t: \Mor(H) \to \Ob(H)$ and identity $e: \Ob(H) \to
\Mor(H)$. 

An {\it $H$-diagram} $X$ in sets can be defined in multiple equivalent ways:
\begin{itemize}
\item[1)] The internal definition: an $H$-diagram $X$ is a
  cosimplicial set map $\pi: X \to Ob(H)$, together with an $H$-action
\begin{equation*}
\xymatrix{
\Mor(H) \times_{s,\pi} X \ar[r]^-{m} \ar[d] & X \ar[d]^{\pi} \\
\Mor(H) \ar[r]_{t} & \Ob(H)
}
\end{equation*}
which respects composition laws and identities of $H$. 

\item[2)]
The $H$-diagram $X$
consists of functors $X^{k}: H^{k} \to \mathbf{Set}$ and natural
transformations $h_{\theta}: X^{m} \to X^{n}\theta$ for each $\theta:
\mathbf{m} \to \mathbf{n}$, such that the usual compatibility conditions
are satisfied. 

The compatibility conditions amount to the following: the
transformation $h_{1}$ associated to an identity morphism is the
identity, and if one is given composable ordinal number maps
\begin{equation*}
\mathbf{m} \xrightarrow{\theta} \mathbf{n} \xrightarrow{\gamma} \mathbf{k}
\end{equation*}
then the diagram of natural transformations
\begin{equation*}
\xymatrix{
X^{m} \ar[r]^{h_{\theta}} \ar[d]_{h_{\gamma\theta}} & X^{n}\theta \ar[d]^{h_{\gamma}\theta} \\
X^{k}(\gamma\theta) \ar[r]_{=} & (X^{k}\gamma)\theta
}
\end{equation*}
commutes. 

\item[3)] Write $E_{\mathbf{\Delta}} H$ for the Grothendieck
  construction of the cosimplicial diagram of groupoids $H$. The
  category $E_{\mathbf{\Delta}} H$ has as objects all pairs
  $(\mathbf{n},x)$ such that $\mathbf{n}$ is an ordinal number and $x$
  is an object of the groupoid $H^{n}$. A morphism $(\gamma,f):
  (\mathbf{n},x) \to (\mathbf{m},y)$ of $E_{\mathbf{\Delta}} H$
  consists of an ordinal number morphism $\gamma: \mathbf{n} \to
  \mathbf{m}$ and a morphism $f: \gamma(x) \to y$ of the groupoid
  $H^{m}$. An $H$-diagram $X$ in sets is a set-valued functor $X:
  E_{\mathbf{\Delta}} H \to \mathbf{Set}$.
\end{itemize}

In the internal functor description 1), the elements of $\Mor(H)^{n}
\times_{s,t} X^{n}$ over an object $i$ of $H^{n}$ are pairs
$(\alpha,x)$ such that $\alpha: i \to j$ is a morphism of $H^{n}$ and
$x$ is a member of the fibre $X^{n}(i)$ over $i$ of the map $X^{n} \to
\Ob(H^{n})$. Then $m(\alpha,x) = \alpha_{\ast}(x) \in X(j)$ defines
the corresponding functor $X^{n}: H^{n} \to \mathbf{Set}$ in description 2).
The transformations $h_{\theta}: X^{m}(i) \to X^{n}(\theta(i))$ in
description 2) correspond to the functions $\theta: X^{m} \to X^{n}$
in the commutative diagrams
\begin{equation}\label{eq 3}
\xymatrix{
X^{m} \ar[r]^{\theta} \ar[d]_{\pi} & X^{n} \ar[d]^{\pi} \\
\Ob(H^{m}) \ar[r]_{\theta} & \Ob(H^{n})
}
\end{equation}
by restriction to fibres.

The diagram (\ref{eq 3})
is the simplicial degree $0$ part of the
commutative diagram
\begin{equation*}
\xymatrix{
\hocolim_{H^{m}}\ X^{m} \ar[r]^{\theta} \ar[d] 
& \hocolim_{H^{n}}\ X^{n} \ar[d] \\
BH^{m} \ar[r]_{\theta} & BH^{n}
}
\end{equation*}
of simplicial set maps which arises from description 2).
The respective homotopy colimits define a cosimplicial space
$\hocolim_{H} X$ and a canonical cosimplicial space map $\hocolim_{H}
X \to BH$.

The homotopy colimit $\hocolim_{H^{n}}X^{n}$ is the standard
Bousfield-Kan homotopy colimit. It is the nerve $B(E_{H^{n}}X^{n})$ of
the translation category $E_{H^{n}}X^{n}$ for the functor $X^{n}:
H^{n} \to \mathbf{Set}$. The objects of this
category are pairs $(i,x)$ with $i \in \Ob(H^{n})$ and $x \in
X^{n}(i)$, and its morphisms $(i,x) \to (j,y)$ consist of pairs
$(\alpha,f)$ such that $\alpha: i \to j$ is a morphism of $H^{n}$ and
$f: X^{n}(\alpha)(x) \to y$ is a function.

The corresponding internally defined functor
$X \to \Ob(H)$ is the part of the simplicial cosimplicial set map
$\hocolim_{H} X \to BH$ in simplicial degree $0$, the identities are defined by 
the degeneracy
\begin{equation*}
s_{0}: X \to (\hocolim_{H} X)_{1} =  \Mor(H) \times_{s,\pi} X,
\end{equation*}
and the multiplication map 
\begin{equation*}
m: \Mor(H) \times_{s,\pi} X \to X
\end{equation*}
 is the face map $d_{0}$. The requirement that the multiplication map
 for the internal functor respects laws of composition amounts to the
 simplicial identity $d_{0}d_{1}=d_{0}d_{0}$, and 
 multiplication respects identities by the relation $d_{1}s_{0} =
 d_{1}s_{1} = 1$.
\medskip

An $H$-diagram $X$ in sets is said to be an $H$-{\it torsor} if the cosimplicial
space $\hocolim_{H}\ X$ is weakly
equivalent to the terminal object $\ast$. A {\it morphism of $H$-torsors}
$f: X \to Y$ is a natural transformation
\begin{equation*} 
\xymatrix@C=10pt{
X \ar[rr]^{f} \ar[dr]_{\pi} && Y \ar[dl]^{\pi} \\
& \Ob(H)
}
\end{equation*}
in the usual sense.

The diagram of cosimplicial spaces
\begin{equation*}
\xymatrix{
X \ar[r] \ar[d] & \hocolim_{H}\ X \ar[d] \\
\Ob(H) \ar[r] & BH
}
\end{equation*}
is sectionwise homotopy cartesian for each $H$-diagram $X$, by the technology
around Quillen's Theorem B --- see
\cite[IV.5.7]{GJ}. It follows that if $f: X \to Y$ is a map of
$H$-torsors, then the map $f: X \to Y$ of cosimplicial sets is a weak
equivalence of cosimplicial spaces, and is therefore an
isomorphism. The category $H-\mathbf{tors}$ of $H$-torsors and natural
transformations between them therefore forms a groupoid.

The functor 
\begin{equation}\label{eq 4}
\hocolim_{H}: H-\mathbf{tors} \to h(\ast,BH)
\end{equation}
 takes
an $H$-torsor $X$ to its {\it canonical cocycle}
\begin{equation*}
\ast \xleftarrow{\simeq} \hocolim_{H}\ X \to BH
\end{equation*}
in cosimplicial spaces. The canonical cocycle functor has a left
adjoint
\begin{equation}\label{eq 5}
\pb: h(\ast,BH) \to H-\mathbf{tors}
\end{equation}
 which is defined in sections by taking path components of pullbacks
 along the canonical maps $B(H^{n}/x) \to BH^{n}$ (see \cite{J39},
 \cite[Lem 9.16]{LocHom}), and we have the following:

\begin{theorem}\label{th 8}
There are induced isomorphisms
\begin{equation*}
\pi_{0}(H-\mathbf{tors}) \cong \pi_{0}h(\ast,BH) \cong [\ast,BH]
\end{equation*}
\end{theorem}

\noindent
Again, $[\ast,BH]$ denotes the set of morphisms in the homotopy category of
cosimplicial spaces, and is also isomorphic to the set
$\pi_{0}\Tot(BH)$. 

Theorem \ref{th 8} is a special case of a
principle which identifies non-abelian cohomology with isomorphism
classes of torsors. The torsors that are considered
here are torsors for groupoids --- this concept is a generalization of
torsors for groups, or classical principal homogeneous spaces.

\begin{remark}
Every torsor $X$ for a cosimplicial groupoid $H$ consists of functors
$X^{n}: H^{n} \to \mathbf{Set}$, which are themselves torsors in the
sense that the simplicial set maps $\hocolim_{H^{n}} X^{n} \to \ast$
are weak equivalences. The simplicial set $\hocolim_{H^{n}} X^{n}$ is
the nerve of the translation groupoid $E_{H^{n}}X^{n}$ which has
objects consisting of pairs $(x,v)$ with $x \in X^{n}(v)$. The set of
objects of $E_{H^{n}}X^{n}$ is non-empty, and so there is a natural transformation
\begin{equation*}
H^{n}(\ ,v) \xrightarrow{x} X^{n}
\end{equation*}
of functors defined on the groupoid $H^{n}$. This natural
transformation is a map of $H^{n}$-torsors, and is therefore an
isomorphism. This transformation is a {\it trivialization} of the torsor
$X^{n}$ in the geometric sense.

The collection of $H$-torsors therefore consists of functors $X
\to \Ob(H)$ such that $X$ has cardinality bounded above by $\vert
\Mor(H) \vert$. We can thus assume that the groupoid
$H-\mathbf{tors}$ is a small groupoid, and so the nerve
$B(H-\mathbf{tors})$ is a simplicial set.

There is a way \cite[Prop 6.7]{LocHom} to replace the cocycle category
$h(\ast,BH)$ by a small category up to ``weak equivalence'', but
we shall not need this here.
\end{remark}

\begin{example}
There are cosimplicial groupoids $H$ for which the associated
cosimplicial space $BH$ is not Bousfield-Kan fibrant.

To see this, observe first of all that if $f: K \to H$ is a morphism
between contractible groupoids, then the induced map $f: BK \to BH$ is
a fibration if and only if $f$ is surjective on objects. 

If $H$ is a cosimplicial contractible groupoid, then all of the groupoids
$M^{n}H$ are contractible: if $(x_{0}, \dots ,x_{n})$ and $(y_{0}, \dots
,y_{n})$ are objects of $M^{n}H$ then there is a unique morphism
$f_{i}:x_{i} \to y_{i}$ in $H^{n-1}$, and these morphisms $f_{i}$ are
consistent with the cosimplicial identities because they specialize to
unique morphisms of $H^{n-2}$ under the codegeneracy maps. It follows
that the maps 
\begin{equation*}
s: BH^{n} \to M^{n-1}BH = B(M^{n-1}H)
\end{equation*} 
are fibrations
if and only if the functors $s: H^{n} \to M^{n-1}H$ are surjective on
objects.

There are cosimplicial sets $X$ for which the functions $s: X^{n} \to
M^{n-1}X$ are not surjective in general. In the cosimplicial category
$\mathbf{\Delta}$, the category $M^{1}\mathbf{\Delta} \cong \mathbf{1}
\times \mathbf{1}$ has four objects, so the functor
\begin{equation*}
s: \mathbf{2} \to M^{1}\mathbf{\Delta}
\end{equation*}
cannot be surjective on objects. 

Suppose that $X$ is a cosimplicial set, and let $C(X)$ be the
degreewise contractible groupoid on $X$. Then the groupoid $C(X)^{n}$ has
the set $X^{n}$ as objects, and has exactly one morphism between any
two elements of $X^{n}$. The groupoid $M^{n}C(X)$ is the contractible
groupoid on the set $M^{n}X$ for all $n$, and so there are
cosimplicial sets $X$ for which the functors $s: C(X)^{n} \to
M^{n-1}C(X)$ are not all surjective on objects.
\end{example}

The cosimplicial space $BH$ for a cosimplicial groupoid $H$ is not
Bous\-field-Kan fibrant in general, but we can form the function complex
$\mathbf{hom}(\Delta,BH)$. This object is the nerve of a groupoid
$H^{\mathbf{\Delta}}$ whose objects are all cosimplicial functors
$\mathbf{\Delta} \to H$ and whose morphisms are the natural
transformations of these functors.

\begin{lemma}\label{lem 11}
Suppose that $U$ is a cosimplicial groupoid such that the map $BU \to
\ast$ is a weak equivalence of cosimplicial spaces. Then
the function complex $\mathbf{hom}(\Delta,BU)$ is a contractible space.
\end{lemma}

\begin{proof}
One shows that there is an isomorphism of groupoids
\begin{equation*}
U^{\mathbf{\Delta}} \cong U^{0},
\end{equation*}
while $U^{0}$ is a contractible groupoid by assumption.

The groupoid $U^{0}$ is non-empty.  Pick $a \in U^{0}$ and let it
define a functor $\mathbf{0} \xrightarrow{a} U^{0}$. The image of the
vertex $i \in \mathbf{n}$ in $U^{n}$ is determined by the composite
\begin{equation*}
\mathbf{0} \xrightarrow{a} U^{0} \xrightarrow{i_{\ast}} U^{n},
\end{equation*}
where $i: \mathbf{0} \to \mathbf{n}$ is the ordinal number morphism
which picks out the vertex $i$. A functor $a:\mathbf{n} \to U^{n}$ is
defined by sending the morphism $i \leq j$ to the unique morphism
$i_{\ast}(a) \to j_{\ast}(a)$ of the groupoid $U^{n}$. The functors
$a: \mathbf{n} \to U^{n}$, $n \geq 0$, define a map $\mathbf{\Delta}
\to U$ of cosimplicial categories.  Conversely, a morphism
$\mathbf{\Delta} \to U$ is completely determined by the object
$\mathbf{0} \to U^{0}$ in cosimplicial degree $0$.

The groupoids
$(U^{n})^{\mathbf{1}}$ of morphisms in $U^{n}$ are contractible, so a
morphism of the groupoid $U^{\Delta}$ is completely determined by the
part in cosimplicial degree $0$.
\end{proof}

Recall that a map $G \to G'$ of groupoids is a
weak equivalence if and only if the induced map $BG \to BG'$ of
classifying spaces is a weak equivalence.

There is a model structure on the category
$\mathbf{Gpd}^{\mathbf{\Delta}}$ of the cosimplicial group\-oids for
which the weak equivalences (respectively fibrations) are those maps
$f: G \to H$ for which the induced maps $BG \to BH$ are weak
equivalences (respectively injective fibrations) of cosimplicial
spaces. This is a special case of general results about sheaves and/or
presheaves of groupoids, for which the usual references are \cite{JT0}
and \cite{H}. 

This model structure has an associated definition of cocycles and
cocycle categories $h(G,H)$ in cosimplicial groupoids: a cocycle in
this category is a diagram
\begin{equation*}
G \underset{\simeq}{\xleftarrow{g}} K \to H
\end{equation*}
in cosimplicial groupoids for which the map $g$ is a weak equivalence.

A cocycle
\begin{equation*}
\ast \xleftarrow{\simeq} V \to BH
\end{equation*}
in cosimplicial spaces defines a cocycle
\begin{equation*}
\ast \xleftarrow{\simeq} \pi(V) \to H
\end{equation*}
in cosimplicial groupoids, by an adjointness argument, where $\pi(V)$
is the result of applying the fundamental groupoid functor in all
cosimplicial degrees. The fundamental groupoid functor therefore
defines a functor of cocycle categories
\begin{equation*}
\pi: h(\ast,BG) \xrightarrow{\cong} h(\ast,G).
\end{equation*}
This functor has a right adjoint
\begin{equation*}
B: h(\ast,G) \to h(\ast,BG)
\end{equation*}
which sends a cocycle 
\begin{equation*}
\ast \xleftarrow{\simeq} H \to G
\end{equation*}
 in
cosimplicial groupoids to the cocycle 
\begin{equation*}
\ast \xleftarrow{\simeq} BH \to
BG
\end{equation*} 
in cosimplicial spaces.  It follows that there are isomorphisms
\begin{equation*}
\pi_{0}h(\ast,G) \cong \pi_{0}h(\ast,BG) \cong [\ast,BG].
\end{equation*}

Make a fixed choice of morphism $x_{U}: \mathbf{\Delta} \to U$ for
all cosimplicial groupoids $U$ such that $U \to \ast$ is a weak
equivalence, as in the proof of Lemma \ref{lem 11}.

Suppose given a cocycle
\begin{equation*}
\ast \xleftarrow{\simeq} U \xrightarrow{f} H
\end{equation*}
in cosimplicial groupoids. We have our fixed choice of morphism $x_{U}:
\mathbf{\Delta} \to U$ of cosimplicial categories. Write $a_{U}$ for
the composite $f\cdot x_{U}$.

Suppose that $g: U \to V$ is a morphism of cocycles, and consider
the diagram
\begin{equation*}
\xymatrix@C=10pt{
& U \ar[rr]^{g} \ar[dr] && V \ar[dl] \\
\mathbf{\Delta} \ar[ur]^{x_{U}} \ar[rr]_{a_{U}} && H 
&& \mathbf{\Delta} \ar[ul]_{x_{V}} \ar[ll]^{a_{V}}
}
\end{equation*}
Then $V^{\mathbf{\Delta}}$ is a contractible groupoid by Lemma
\ref{lem 11}, and so there is a unique natural transformation $g\cdot
x_{U} \to x_{V}$, which induces a morphism $g_{\ast}: a_{U} \to a_{V}$
in $H^{\mathbf{\Delta}}$.

The assignment $g \mapsto g_{\ast}$ is functorial. We have therefore
defined a functor
\begin{equation*}
s: h(\ast,H) \to H^{\mathbf{\Delta}}.
\end{equation*}
The functor
\begin{equation*}
\hocolim_{H}: H-\mathbf{tors} \to h(\ast,BH)
\end{equation*}
factors through (and is defined by) a functor
\begin{equation*}
E_{H}: H-\mathbf{tors} \to h(\ast,H),
\end{equation*}
where $(E_{H}X)^{n}$ for a torsor $X$ is the translation groupoid
which is associated to the functor $X^{n}: H^{n} \to \mathbf{Set}$.

\begin{theorem}\label{th 12}
The composite functor
\begin{equation}\label{eq 6}
H-\mathbf{tors} \xrightarrow{E_{H}} h(\ast,H) \xrightarrow{s} H^{\mathbf{\Delta}}
\end{equation}
is a weak equivalence of groupoids.
\end{theorem}

\begin{proof}
Suppose that $X$ is an $H$-torsor. Following the choices made above,
write $x = x_{E_{H}X}$ and $a = a_{E_{H}X}$. Consider the functor $x:
\mathbf{n} \to E_{H^{n}}X^{n}$ in cosimplicial degree $n$. Then for
each $i \in \mathbf{n}$, $x(i) = (a(i),x_{i})$ with $x_{i} \in
X(a(i))$, and there is an induced isomorphism of $H^{n}$-torsors
\begin{equation*}
H^{n}(\ ,a(i)) \underset{\cong}{\xrightarrow{x_{i}}} X^{n}.
\end{equation*} 

If $i \leq j$, then the diagram
\begin{equation*}
\xymatrix@R=10pt{
H^{n}(\ ,a(i)) \ar[dr]^{x_{i}} \ar[dd]_{\alpha_{\ast}} \\
& X^{n} \\
H^{n}(\ ,a(j)) \ar[ur]_{x_{j}}
}
\end{equation*}
commutes, where $\alpha: a(i) \to a(j)$ is defined by the functor $a$.

Suppose that $g: X \to Y$ is a morphism of $H$-torsors. Then all diagrams 
\begin{equation*}
\xymatrix{
H^{n}(\ ,a(i)) \ar[r]^-{x_{i}}_-{\cong} \ar[d]_{s(g_{\ast})} & X^{n} \ar[d]^{g} \\
H^{n}(\ ,b(i)) \ar[r]_-{y_{i}}^-{\cong} & Y^{n}
}
\end{equation*}
commute, where $y = x_{E_{H}Y}$ and $b = a_{E_{H}Y}$.
It follows that the morphism of torsors $g$ is completely
determined by the morphism $s(g_{\ast})$ of $H^{\mathbf{\Delta}}$. The composite
functor (\ref{eq 6}) is therefore fully faithful.

If $\tau: \mathbf{\Delta} \to H$ is an
object of $H^{\mathbf{\Delta}}$ then $\tau$ is a
cocycle, and there is a cocycle morphism
\begin{equation*}
\xymatrix{
\mathbf{\Delta} \ar[r]^-{\eta} \ar[dr]_{\tau} & E_{H} \pb(\tau) \ar[d] \\
& H
}
\end{equation*}
for an $H$-torsor $\pb(\tau)$ which arises from the adjunction of
(\ref{eq 4}) and (\ref{eq 5}). If $x = x_{E_{H}\pb(\tau)}:
\mathbf{\Delta} \to E_{H} \pb(\tau)$ with image $a: \mathbf{\Delta}
\to H$, then there is a unique morphism $\eta \to x$ in the trivial
groupoid $(E_{H} \pb(\tau))^{\mathbf{\Delta}}$, whose image is a
morphism $\tau \to a$ of $H^{\mathbf{\Delta}}$.
\end{proof}

\begin{corollary}\label{cor 13}
The composite functor (\ref{eq 6}) induces a weak equivalence
\begin{equation*}
B(H-\mathbf{tors}) \xrightarrow{\simeq} 
B(H^{\mathbf{\Delta}}) = \mathbf{hom}(\Delta,BH).
\end{equation*}
\end{corollary}

The functor 
\begin{equation*}
H \mapsto B(H-\mathbf{tors})
\end{equation*}
takes weak equivalences of cosimplicial groupoids to weak
equivalences of spaces, so we also have the following:

\begin{corollary}\label{cor 14}
Any weak equivalence $f: G \to H$ of cosimplicial groupoids induces a
weak equivalence
\begin{equation*}
f_{\ast}: \mathbf{hom}(\Delta,BG) \to \mathbf{hom}(\Delta,BH).
\end{equation*}
\end{corollary}

We say that a cosimplicial groupoid $G$
is a {\it stack} if the cosimplicial space $BG$ is
injective fibrant. The injective model structure for cosimplicial
groupoids is cofibrantly generated, so there is a functorial fibrant
model $j: G \to St(G)$ for a cosimplicial groupoid $G$, such that the
map $j$ is a trivial cofibration and $St(G)$ is injective
fibrant. This fibrant model $St(G)$ is a functorial stack completion for
$G$. More generally, a weak equivalence $G \to H$ of cosimplicial
groupoids such that $H$ is a stack is called a {\it stack completion}
of $G$.

If $G$ is a cosimplicial groupoid and $j: G \to H$ is a stack
completion of $G$, then Corollary \ref{cor 14} implies that the
induced map
\begin{equation*}
\mathbf{hom}(\Delta,BG) \xrightarrow{j_{\ast}} \mathbf{hom}(\Delta,BH)
\end{equation*}
is a weak equivalence. At the same time, the weak equivalence $\Delta
\to \ast$ induces a weak equivalence
\begin{equation*}
\Tot(BH) \cong \mathbf{hom}(\ast,BH) \to \mathbf{hom}(\Delta,BH)
\end{equation*}
since $BH$ is injective fibrant. Thus, in sheaf theoretic language,
the groupoid $G^{\mathbf{\Delta}}$  is equivalent to the groupoid of 
global sections of
the stack completion of the cosimplicial groupoid $G$.

\section{Abelian cohomology}

The results of this section appear in \cite{BK} for the most part, and
are essentially well known. They are included here for the sake of
completeness, and the overall description is from a sheaf theoretic
point of view.

\begin{lemma}\label{lem 15}
Suppose that $A$ is a cosimplicial abelian group. Then there is an
abelian group homomorphism $j: M^{n-1}A \to A^{n}$ such that the
composite $s\cdot j$ is the identity on $M^{n-1}A$. The homomorphism
$j$ is natural in cosimplicial abelian groups $A$.
\end{lemma}

\begin{proof}
Suppose that $(0, \dots ,0,a_{i}, \dots ,a_{n-1})$ is an element of
$M^{n-1}A$. Then $s^{j}a_{i} = 0$ for $j < i$, and we have
\begin{equation*}
\begin{aligned}
sd^{i+1}a_{i} 
&= (s^{0}d^{i+1}a_{i}, \dots ,s^{i-1}d^{i+1}a_{i},a_{i},a_{i}, \dots) \\
&= (d^{i}s^{0}a_{i}, \dots ,d^{i}s^{i-1}a_{i},a_{i},a_{i}, \dots) \\
&= (0, \dots ,0,a_{i},a_{i}, \dots)
\end{aligned}
\end{equation*}
It follows that 
\begin{equation*}
(0, \dots, 0,a_{i}, \dots, a_{n}) - sd^{i+1}a_{i} = (0, \dots ,0,a_{i+1}-a_{i}, \dots),
\end{equation*}
and we construct $j(a_{0}, \dots a_{n})$ inductively by setting
\begin{equation*}
j(0,\dots,0,a_{i}, \dots ,a_{n}) = d^{i+1}a_{i} + j(0, \dots ,0,a_{i+1}-a_{i}, \dots).
\end{equation*}
\end{proof}

\begin{corollary}\label{cor 16}
Suppose that $f: A \to B$ is a map of cosimplicial objects in
simplicial abelian groups such that each morphism $f: A^{n} \to B^{n}$
is a fibration of simplicial abelian groups. Then $f$ is a
Bousfield-Kan fibration.
\end{corollary}

\begin{proof}
We use the Dold-Kan correspondence \cite[III.2.3]{GJ},
\cite[III.2.11]{GJ} to suppose that $p: A \to B$ is a morphism of
cosimplicial objects in chain complexes such that each chain map $f:
A^{n} \to B^{n}$ is surjective in non-zero degrees.

Suppose that $m > 0$ and consider the map
\begin{equation*}
A^{n+1}_{m} \to B^{n+1}_{m} \times_{M^{n}B_{m}} M^{n}A_{m}.
\end{equation*}
We want to show that this homomorphism is surjective.

For this, it suffices to show that the indicated map $p_{\ast}$ in the
comparison of exact sequences
\begin{equation*}
\xymatrix{
0 \ar[r] & K_{1} \ar[r] \ar[d]_{p_{\ast}} & A^{n+1}_{m} \ar[r]^{s} \ar[d]^{p} 
& M^{n}A_{m} \ar[r] \ar[d]^{p} & 0 \\
0 \ar[r] & K_{2} \ar[r] & B^{n+1}_{m} \ar[r]_{s} & M^{n}B_{m} \ar[r] & 0 
}
\end{equation*}
is surjective. This map $p_{\ast}$ is a direct summand of the
surjective map $p: A^{n+1}_{m} \to B^{n+1}_{m}$ by Lemma \ref{lem 15},
and is therefore surjective.
\end{proof}

\begin{corollary}\label{cor 17}
Suppose that $A$ is a cosimplicial object in simplicial abelian
groups. Then there are isomorphisms
\begin{equation*}
[\ast,A] \cong \pi_{0}\mathbf{hom}(\Delta,A) \cong \pi_{ch}(N\mathbb{Z}\Delta,NA),
\end{equation*}
where $N$ is the normalized chains functor and $\pi_{ch}(\ ,\ )$ is
chain homotopy classes of maps in cosimplicial chain complexes.
\end{corollary}

\begin{proof}
The cosimplicial object $A$ in simplicial abelian groups is
Bousfield-Kan fibrant by Corollary \ref{cor 16}, and so the group of
homotopy classes of maps $\Delta \to A$ coincides with the group of morphisms
$[\Delta,A] \cong [\ast,A]$ in the homotopy category of cosimplicial
spaces.

Chain homotopy is defined by a natural path object for chain
complexes, which therefore defines a path object for cosimplicial
objects in simplicial abelian groups through the Dold-Kan
correspondence, again by Corollary \ref{cor 16}. It follows that the
morphisms $\Delta \to A$ are homotopic if and only if the
corresponding morphisms $N\mathbb{Z}\Delta \to NA$ are chain
homotopic.
\end{proof}

Now suppose that $A$ is a cosimplicial abelian group. For $k \leq n-1$,
let $M_{k}^{n-1}A$ be the set of tuples $(a_{0}, \dots ,a_{k})$ with
$a_{i} \in A^{n-1}$ and $s^{i}a_{j} = s^{j-1}a_{i}$ for $i<j$. There
is a natural map $s: A^{n} \to M_{k}^{n-1}A$ which is defined by
\begin{equation*}
s(a) = (s^{0}a,s^{1}a, \dots ,s^{k}a).
\end{equation*}
This morphism $s$ has a natural splitting and is therefore surjective,
as in the proof of Lemma \ref{lem 15}. 

Write $cN_{k}A^{n}$ for the kernel of the map $s: A^{n} \to
M_{k}^{n-1}A$. Then $cN_{k}A^{n}$ is the intersection of the kernels of
the $s^{i}: A^{n} \to A^{n-1}$, for $0 \leq i \leq k$.

The coboundary 
\begin{equation*}
\delta = \sum_{i=0}^{n} (-1)^{i}d^{i}: A^{n} \to A^{n+1}
\end{equation*}
induces a morphism $cN_{k}A^{n} \to cN_{k}A^{n+1}$. We therefore have
a natural cochain inclusion $cN_{k}A \subset A$. Write $cNA$ for the
intersection of the complexes $cN_{k}A$ in $A$.

\begin{lemma}\label{lem 18}
Suppose that $A$ is a cosimplicial abelian group.
Then the cochain complex map $cN_{k}A \subset A$ is a cohomology isomorphism
for all $k$. The inclusion $cNA \subset A$ is also a cohomology isomorphism.
\end{lemma}

\begin{proof}
There are short exact sequences of cochain complexes
\begin{equation*}
0 \to cN_{k+1}A^{n} \to cN_{k}A^{n} \xrightarrow{s_{\ast}} C^{n} \to 0
\end{equation*}
where $C^{n} = 0$ if $k \geq n-1$ and the map $s_{\ast}$ is the map
\begin{equation*}
s^{k+1}: cN_{k}A^{n} \to cN_{k}A^{n-1}
\end{equation*}
if $k < n-1$.  In the latter case, the map $s^{k+1}$ has a section given
by $d^{k+2}$.

Suppose that $n>k+1$ and form the diagram
\begin{equation*}
\xymatrix{
0 \ar[r] & cN_{k+1}A^{n} \ar[r] \ar[d]_{\delta} 
& cN_{k}A^{n} \ar[r]^{s^{k+1}} \ar[d]^{\delta} 
& cN_{k}A^{n-1} \ar[r] \ar[d]^{\delta_{\ast}} & 0 \\
0 \ar[r] & cN_{k+1}A^{n+1} \ar[r] & cN_{k}A^{n+1} \ar[r]_{s^{k+1}} & cN_{k}A^{n} \ar[r] & 0
}
\end{equation*}
Then 
\begin{equation*}
s^{k+1}(\sum_{j=0}^{n+1} (-1)^{j}d^{j})(x) = (\sum_{j=k+3}^{n+1} (-1)^{j}d^{j-1}s^{k+1})(x)
\end{equation*}
for $x \in cN_{k}A^{n}$, so that
\begin{equation*}
\delta_{\ast} = \sum_{j=k+2}^{n} (-1)^{j+1}d^{j}
\end{equation*}
for $n > k+1$. The cochain complex $C^{\ast}$ has a contracting
homotopy defined by the maps $s_{k+1}$ in degrees where it is
non-zero.

It follows that all inclusions $cN_{k+1}A \subset cN_{k}A$ are
cohomology isomorphisms. 

A similar argument shows that the inclusion $cN_{0}A \subset A$ is a
cohomology isomorphism. The quotient complex for this
inclusion is isomorphic to the cochain complex
\begin{equation*}
0 \to A^{0} \xrightarrow{d^{1}} A^{1} \xrightarrow{d^{1}-d^{2}} A^{2} \xrightarrow{d^{1}-d^{2}+d^{3}} A^{3} \to \dots,
\end{equation*}
and the quotient map is the map $s^{0}$ in positive degrees. The
contracting homotopy is the map $s^{1}$.
\end{proof}

Suppose again that $A$ is a cosimplicial abelian group, and form the
cosimplicial space $K(A,n)$, as a cosimplicial object in
simplicial abelian groups.

\begin{lemma}\label{lem 19}
Suppose that $A$ is a cosimplicial abelian group. Then there is a
natural isomorphism
\begin{equation*}
\pi_{0}\mathbf{hom}(\Delta,K(A,n)) \cong H^{n}(A),
\end{equation*}
where $H^{n}(A)$ is the $n^{th}$ cohomology group of the cochain
complex associated to $A$.
\end{lemma}

\begin{proof}
A cosimplicial chain map $f: N\mathbb{Z}\Delta \to A[-n]$ is uniquely
specified by the chain complex morphisms
\begin{equation*}
f^{k}: N\mathbb{Z}\Delta^{k} \to A^{k}[-n]
\end{equation*}
for $k \geq n$, which morphisms respect the cosimplicial identities. These
cochain complex morphisms are completely determined by the morphism
$f^{n}$, and in particular the image $f(\iota_{n}) \in A^{n}$ of the
classifying simplex. The requirement that the diagram
\begin{equation*}
\xymatrix{
N\mathbb{Z}\Delta^{n+1} \ar[r]^{f^{n+1}} & A^{n+1}[-n] \\
N\mathbb{Z}\Delta^{n} \ar[r]^{f^{n}} \ar[u]^{d^{i}} \ar[d]_{s^{j}} & A^{n}[-n] \ar[u]_{d^{i}} \ar[d]^{s^{j}} \\
N\mathbb{Z}\Delta^{n-1} \ar[r]_{f^{n-1}=0} & A^{n-1}[-n] 
}
\end{equation*}
commutes forces $f(\iota_{n}) \in cNA^{n}$ and
$\sum_{i=0}^{n+1}(-1)^{i}d^{i}f(\iota_{n}) = 0$. Conversely, a cycle $z
\in cNA^{n}$ completely determines the map $f$.

Similarly, a cosimplicial chain homotopy $s$
between chain maps $f,g: N\mathbb{Z}\Delta^{n} \to A[n]$ is defined by
the element $s(\iota_{n-1}) \in cNA^{n-1}$ such that 
\begin{equation*}
\delta s(\iota_{n-1}) = f(\iota_{n})-g(\iota_{n})
\end{equation*}
in $cNA^{n}$. 

It follows that there is an isomorphism
\begin{equation*}
\pi_{ch}(N\mathbb{Z}\Delta,A[-n]) \cong H^{n}(cNA)
\end{equation*}
which is natural in cosimplicial abelian groups $A$.
There are natural isomorphisms
\begin{equation*}
[\ast,K(A,n)] \cong \pi_{ch}(N\mathbb{Z}\Delta,A[-n]) 
\end{equation*}
and
\begin{equation*}
H^{n}(cNA) \cong H^{n}(A)
\end{equation*}
by Corollary \ref{cor 17} and Lemma \ref{lem 18}, respectively.
\end{proof}

\begin{corollary}\label{cor 20}
Suppose that $A$ is a cosimplicial abelian group. Then there are
natural isomorphisms
\begin{equation*}
\pi_{k}\mathbf{hom}(\Delta,K(A,n)) \cong
\begin{cases}
H^{n-k}(A) & \text{if $0 \leq k \leq n$} \\
0 & \text{if $k>n$.}
\end{cases}
\end{equation*}
\end{corollary}

Suppose that $i: A \to J^{\ast}$ is an injective resolution of $A$ in the
category of cosimplicial abelian groups, thought of as a morphism of
unbounded chain complexes with $A$ concentrated in chain degree $0$. Then
there is a induced weak equivalence of cosimplicial chain complexes
\begin{equation*}
i: A[-n] \to \Tr(J^{\ast}[-n])
\end{equation*}
where $\Tr$ is the good truncation functor in degree $0$ (which
preserves homology isomorphisms). Write $K(J,n)$ for the cosimplicial
object in simplicial abelian groups which is given by applying the
Dold-Kan correspondence to $\Tr(J^{\ast}[-n])$. Then there is an
induced weak equivalence
\begin{equation*}
i: K(A,n) \to K(J,n).
\end{equation*}

The weak equivalences $i$ and $\Delta \to \ast$ induce isomorphisms
\begin{equation*}
\pi_{ch}(N\mathbb{Z}\Delta,A[-n]) \cong
\pi_{ch}(N\mathbb{Z}\Delta,\Tr(J^{\ast}[-n])) \cong
\pi_{ch}(N\mathbb{Z}\ast,\Tr(J^{\ast}[-n])).
\end{equation*}
The first of these isomorphisms is a consequence of 
Corollary \ref{cor 17}.  

For the latter, a cosimplicial space $X$ defines a
bicomplex $\hom(X_{n},J^{p})$ with associated spectral sequence
\begin{equation}\label{eq 7}
E_{2}^{p,q} = \Ext^{q}(H_{p}X,A) \Rightarrow \pi_{ch}(\mathbb{Z}X,J^{\ast}[-p-q]) = \pi_{ch}(\mathbb{Z}X,\Tr(J^{\ast}[-p-q])).
\end{equation}
It follows that any weak equivalence $X \to Y$ of cosimplicial spaces induces an isomorphism
\begin{equation*}
\pi_{ch}(N\mathbb{Z}Y,\Tr(J^{\ast}[-n])) \xrightarrow{\cong}
\pi_{ch}(N\mathbb{Z}X,\Tr(J^{\ast}[-n]))
\end{equation*}
for all $n$. 

\begin{remark}
The ideas of the last few paragraphs, leading to the ``universal
coefficients'' spectral sequence (\ref{eq 7}) and the displayed
application, are again essentially sheaf theoretic. The spectral
sequence (\ref{eq 7}) is a special case of a spectral sequence which
relates homology sheaves to cohomology groups for simplicial sheaves
and presheaves. Many of these ideas originated in \cite{J5}, and the
theory is discussed in some detail in \cite{LocHom}.
\end{remark}

We have, finally, an isomorphism
\begin{equation*}
\pi_{ch}(N\mathbb{Z}\ast,\Tr(J^{\ast}[-n])) \cong H^{-n}(\varprojlim\ J^{\ast}) = \mathbf{R}\varprojlim{}^{n}(A),
\end{equation*}
where $\mathbf{R}\varprojlim{}^{n}(A)$ is the $n^{th}$ derived
functor of the inverse limit functor on cosimplicial abelian groups.

Lemma \ref{lem 19} therefore implies the following:

\begin{lemma}\label{lem 22}
There is an isomorphism
\begin{equation*}
H^{n}(A) \cong \mathbf{R}\varprojlim{}^{n}(A)
\end{equation*}
which is natural in cosimplicial abelian groups $A$.
\end{lemma}

A cosimplicial abelian group $A$ is the analog of a sheaf of abelian
groups in the present context, and it is a consequence of Lemma
\ref{lem 22} that the groups $H^{n}(A)$ are the corresponding sheaf
cohomology groups. Sheaf cohomology groups are defined to be the
higher derived functors of global sections, while ``global sections''
is a different name for the inverse limit functor.

\section{Postnikov towers}

Suppose that
\begin{equation*}
X_{0} \xleftarrow{p} X_{1} \xleftarrow{p} X_{2} \leftarrow \dots
\end{equation*}
is a tower of sectionwise fibrations of sectionwise fibrant cosimplicial
spaces, and let $X = \varprojlim_{n} X_{n}$. Take an injective fibrant
model
\begin{equation*}
GX_{0} \xleftarrow{q} GX_{1} \xleftarrow{q} GX_{2} \leftarrow \dots
\end{equation*}
of the original tower by first taking an injective fibrant model $j: X_{0}
\to GX_{0}$ ($j$ a trivial cofibration, $GX_{0}$ injective fibrant),
and then inductively form diagrams
\begin{equation*}
\xymatrix{
X_{n+1} \ar[r]^{j} \ar[d]_{p} & GX_{n+1} \ar[d]^{q} \\
X_{n} \ar[r]_{j} & GX_{n}
}
\end{equation*}
such that all maps $j$ are trivial cofibrations and all $q$ are
injective fibrations. Then the diagram of simplicial set maps
\begin{equation*}
\xymatrix{
X^{m}_{n+1} \ar[r]^{j} \ar[d]_{p} & GX^{m}_{n+1} \ar[d]^{q} \\
X^{m}_{n} \ar[r]_{j} & GX^{m}_{n}
}
\end{equation*}
in each cosimplicial degree consists of trivial cofibrations $j$ and
Kan fibrations $p$ and $q$. The maps $j$ form a weak equivalence of
injective fibrant towers in simplicial sets, so that the maps
\begin{equation*}
X^{m} = \varprojlim_{n} X_{n}^{m} \xrightarrow{j_{\ast}}
\varprojlim_{n} GX_{n}^{m}
\end{equation*}
are weak equivalences for all $m$.
It follows that the map
\begin{equation*}
X = \varprojlim_{n} X_{n} \xrightarrow{j_{\ast}}
\varprojlim_{n} GX_{n}
\end{equation*}
is a weak equivalence of cosimplicial spaces. The object $\varprojlim_{n}
GX_{n}$ is injective fibrant, so that $j_{\ast}$ is an injective
fibrant model of $X$.
\medskip

Suppose that $X$ is a cosimplicial Kan complex, and form the Postnikov tower
\begin{equation*}
P_{1}X \xleftarrow{p} P_{2}X \xleftarrow{p} P_{3}X \xleftarrow{p} \dots
\end{equation*}
by making a sectionwise construction. Then the canonical map $X \to
\varprojlim_{n} P_{n}X$ is a weak equivalence, and the composite
\begin{equation*}
X \xrightarrow{\simeq} \varprojlim_{n} P_{n}X \xrightarrow{j_{\ast}} \varprojlim_{n} GP_{n}X
\end{equation*}
is an injective fibrant model of $X$. Write
\begin{equation*}
GX = \varprojlim_{n} GP_{n}X.
\end{equation*}

Suppose, more generally, that $I$ is a small category. It is well
known (see also Remark \ref{rem 4}) that the category of $I$-diagrams
$X: I \to s\mathbf{Set}$, with natural transformations, has an
injective model structure for which the weak equivalences and
cofibrations are defined sectionwise, while the injective fibrations
are defined by a right lifting property. The injective model structure
on cosimplicial spaces is a special case of this object. Again, the
terminal object for the $I$-diagram category is denoted by $\ast$, and
we have cocycles
\begin{equation*}
\ast \xleftarrow{\simeq} U \to X
\end{equation*}
and a cocycle category $h(\ast,X)$ for an $I$-diagram $X$. It is
again a consequence of general results \cite{coc-cat} 
that there is a bijection
\begin{equation*}
\pi_{0}h(\ast,X) \cong [\ast,X]
\end{equation*}
relating path components in the cocycle category and morphisms in the
homotopy category for the injective model structure on $s\mathbf{Set}^{I}$.

We also have the following:

\begin{lemma}\label{lem 23}
Suppose that the $I$-diagram $F$ is a diagram of Eilenberg-Mac Lane
spaces in the sense that there are weak equivalences
$F(i) \simeq K(B(i),n)$ for all objects $i$ of $I$, and for some fixed $n \geq
2$.  Suppose that $F$ has a cocycle 
\begin{equation*}
\ast \xleftarrow{\simeq} U \to F
\end{equation*}
in $I$-diagrams. Then $F$ is weakly equivalent to the
$I$-diagram $K(H_{n}F,n)$.
\end{lemma}

\begin{proof}
Take a factorization
\begin{equation*}
\xymatrix{
U \ar[r]^{i} \ar[dr] & V \ar[d]^{p} \\
& F
}
\end{equation*}
where $i$ is a cofibration and $p$ is a trivial injective
fibration. 
There are maps
\begin{equation*}
F \xleftarrow{p} V \xrightarrow{h} \mathbb{Z}V \to 
\mathbb{Z}V/\mathbb{Z}U \to P_{n}(\mathbb{Z}V/\mathbb{Z}U) 
\xleftarrow{\simeq} K(H_{n},n)
\end{equation*}
where $h$ is the Hurewicz map and $P_{n}$ is the $n^{th}$ Postnikov
section functor in simplicial abelian groups. The composite
\begin{equation*}
V \xrightarrow{h} \mathbb{Z}V \to \mathbb{Z}V/\mathbb{Z}U \to P_{n}(\mathbb{Z}V/\mathbb{Z}U)
\end{equation*}
is a sectionwise weak equivalence by the Hurewicz theorem. The
$I$-diagram $H_{n}$ in abelian groups can be identified with the
integral homology $H_{n}F$ of the diagram $F$.
\end{proof}

\begin{corollary} 
Suppose that $F$ is a diagram of cosimplicial Eilenberg-Mac Lane
spaces in the sense of Lemma \ref{lem 23}. Then $F$ is weakly
equivalent in the $I$-diagram category to $K(A,n)$ for some
cosimplicial abelian group $A$ if and only if $F$ has a cocycle.
\end{corollary}

\begin{proof}
The object $K(A,n)$ has a global base point $\ast \to K(A,n)$ (and
hence a cocycle) which is defined by the element $0$. Thus, if $F$ is weakly
equivalent to $K(A,n)$ then $F$ has a cocycle. The converse assertion
is proved in Lemma \ref{lem 23}.
\end{proof}

Suppose that the cosimplicial space $X$ has a cocycle $U \to
X$, or equivalently that there is a global point $\ast \to GX$ for $GX$.
Let $x: \ast \to GX$ be a choice of base point, and write $x: \ast \to
GP_{n}X$ for its images in the objects $GP_{n}X$. Define the
cosimplicial space $F_{n}$ by the pullback diagram
\begin{equation*}
\xymatrix{
F_{n} \ar[r] \ar[d] & GP_{n}X \ar[d]^{q} \\
\ast \ar[r]_-{x} & GP_{n-1}X
}
\end{equation*}
for $n \geq 2$.
Then $F_{n}$ is injective fibrant, and there is a pullback
\begin{equation*}
\xymatrix{
\mathbf{hom}(\ast,F_{n}) \ar[r] \ar[d] 
& \mathbf{hom}(\ast,GP_{n}X) \ar[d]^{q_{\ast}} \\
\ast \ar[r]_-{x} & \mathbf{hom}(\ast,GP_{n-1}X)
}
\end{equation*}
The space $\mathbf{hom}(\ast,F_{n})$ is non-empty, and it follows from
Lemma \ref{lem 23} that there is a weak equivalence
\begin{equation*}
F_{n} \xrightarrow{\simeq} K(\pi_{n}(GX,x),n).
\end{equation*}

We know how to compute the homotopy groups of the space
$\mathbf{hom}(\ast,F_{n})$ on account of Corollary \ref{cor 20} in the
presence of a cocycle for $X$.  The spectral sequence for the tower of
fibrations
\begin{equation*}
\mathbf{hom}(\ast,GP_{1}X) \xleftarrow{q_{\ast}} \mathbf{hom}(\ast,GP_{2}X) \xleftarrow{q_{\ast}} \mathbf{hom}(\ast,GP_{3}X) \leftarrow \dots 
\end{equation*}
is a special case of the descent spectral sequence for a simplicial
presheaf, with $E_{1}$-terms given by sheaf cohomology groups
\cite{J7}. Thomason's reindexing trick \cite[5.54]{AKTEC} converts
these $E_{1}$-terms to $E_{2}$-terms of the form appearing in the
Bousfield-Kan spectral sequence for the tower of fibrations $\{
\Tot_{s}(GX) \}$ --- see also \cite{GECT}.

\begin{remark}\label{rem 25}
Suppose that $Y$ is a simply connected Kan complex, with Postnikov
tower $\{ P_{n}Y \}$. Form the natural cofibre sequence
\begin{equation*}
P_{n}Y \to P_{n-1}Y \to P_{n-1}Y/P_{n}Y.
\end{equation*}
By this, we mean that we take a functorial replacement of the
fibration $p: P_{n}Y \to P_{n-1}Y$ by a cofibration $P_{n}Y \to A_{n-1}Y$
and then we write $P_{n-1}Y/P_{n}Y$ for the natural fibrant
replacement of the quotient $A_{n-1}Y/P_{n}Y$. Then let
\begin{equation*}
P_{n-1}Y/P_{n}Y \to P_{n+1}(P_{n}Y/P_{n-1}Y) 
\end{equation*}
be the usual fibration
  to the $(n+1)^{st}$ Postnikov section of the homotopy cofibre of $p$.
Then the composite
\begin{equation*}
P_{n-1}Y \to P_{n-1}Y/P_{n}Y \to P_{n+1}(P_{n-1}Y/P_{n}Y)
\end{equation*}
is the $k$-invariant $k_{q}$ of the fibration $q$, and there is a natural
homotopy fibre sequence
\begin{equation}\label{eq 8}
P_{n}Y \xrightarrow{q} P_{n-1}Y \xrightarrow{k_{q}} P_{n+1}(P_{n-1}Y/P_{n}Y)
\end{equation}
which identifies $P_{n}Y$ with the homotopy fibre of $k_{q}$ over the
base point of the diagram $P_{n+1}(P_{n-1}Y/P_{n}Y)$ which is defined by the image
of $P_{n}Y$.

The fibre sequence (\ref{eq 8}) is functorial in simply connected Kan
complexes $Y$.
\end{remark}

Suppose that $H$ is a cosimplicial groupoid. For the discussion of
weak equivalences that appears below, an $H$-diagram in simplicial
sets is best viewed as a functor
$Y:E_{\mathbf{\Delta}} H \to s\mathbf{Set}$ which is defined on the
Grothendieck construction for $H$.

There are functors
$H^{n} \to E_{\mathbf{\Delta}} H$ which are defined by $x \mapsto (\mathbf{n},x)$,
and the functors $Y^{n}$ associated to the functor $Y$ are the
composites
\begin{equation*}
H^{n} \to E_{\mathbf{\Delta}} H \xrightarrow{Y} \mathbf{Set}.
\end{equation*}
The spaces $\hocolim_{H^{n}} Y^{n}$ define a cosimplicial space
$\hocolim_{H} Y$ and a canonical map $\hocolim_{H}Y \to BH$. In this
way we define a functor
\begin{equation*}
\hocolim_{H}: s\mathbf{Set}^{E_{\mathbf{\Delta}} H} \to s\mathbf{Set}^{\mathbf{\Delta}}/BH
\end{equation*}
from the category of $E_{\mathbf{\Delta}}H$ diagrams in simplicial
sets to the category of cosimplicial spaces $Z \to BH$ over $BH$.
Conversely, starting with a cosimplicial space map $X \to BH$, we form
the pullbacks
\begin{equation*}
\xymatrix{
\pb(X)_{x} \ar[r] \ar[d] & X^{n} \ar[d] \\
B(H^{n}/x) \ar[r] & BH^{n}
}
\end{equation*}
for each object $(\mathbf{n},x)$ of the Grothendieck construction
$E_{\mathbf{\Delta}} H$. Then the simplicial sets $\pb(X)_{x}$
define a functor $\pb(X): E_{\mathbf{\Delta}} H \to
s\mathbf{Set}$. The construction is plainly functorial in objects $X
\to BH$, and so we have a functor
\begin{equation*}
\pb: s\mathbf{Set}^{\mathbf{\Delta}}/BH \to s\mathbf{Set}^{E_{\mathbf{\Delta}} H}.
\end{equation*}
The pullback functor $\pb$ is left adjoint to the homotopy colimit
functor $\hocolim_{H}$ since $H$ is a cosimplicial groupoid.
The homotopy colimit functor preserves weak equivalences
by standard nonsense, while the pullback functor preserves weak
equivalences by a Quillen Theorem B argument.  The pullback functor
also preserves cofibrations, so the pullback and homotopy colimit
functors form a Quillen adjunction. This adjunction is a Quillen
equivalence, since the counit map $\epsilon$ and unit map $\eta$ are
weak equivalences for all objects in the respective categories. See
also \cite[Lemma 18]{J39}.

The homotopy colimit functor $\hocolim_{H}$
preserves homotopy cartesian diagrams, since it preserves weak
equivalences and is the right adjoint part of a Quillen adjunction.

Suppose that $X$ is a cosimplicial Kan complex.
There are canonical maps
\begin{equation*}
\xymatrix@C=6pt{
P_{n}X \ar[rr]^{q} \ar[dr]  && P_{n-1}X \ar[dl] \\
& P_{1}X \ar[d]^{j}_{\simeq} \\
& BH
}
\end{equation*}
for $n \geq 2$, where $H$ is a stack completion of the fundamental groupoid 
\begin{equation*}
\pi(P_{1}X) \cong \pi(X) 
\end{equation*}
of $X$. Let $j: \pi(X) \to H$ also denote the map induced by $j:
P_{1}X \to BH$.

Take $x \in H^{m}$. There is a morphism $\alpha: x \to j(y)$ for some
$y \in \pi(X^{m})$, and so there is a weak equivalence
\begin{equation}\label{eq 9}
\pb(P_{n}(X))_{x} \simeq \pb(P_{n}(X))_{y}
\end{equation}
where the latter space is determined by the pullback diagram
\begin{equation*}
\xymatrix{
\pb(P_{n}X)_{y} \ar[r] \ar[d] & P_{n}X^{m} \ar[d] \\
B(\pi(X^{m})/y) \ar[r] & B\pi(X^{m})
}
\end{equation*}
The weak equivalence (\ref{eq 9}) is a special case of a weak equivalence
\begin{equation*}
\pb(Z)_{x} \simeq \pb(Z)_{y}
\end{equation*}
which is defined for and natural in objects $Z \to B\pi(X)$. It follows 
that the map 
\begin{equation*}
q_{\ast}: \pb(P_{n}X)_{x} \to
\pb(P_{n-1}X)_{x}
\end{equation*} 
is weakly equivalent to the map 
\begin{equation}\label{eq 10}
q_{\ast}:
\pb(P_{n}X)_{y} \to \pb(P_{n-1}X)_{y}.
\end{equation}

 The pullback $\pb(P_{n}X)_{y}$ is naturally weakly equivalent to the
 Postnikov section $P_{n}\pb(X)_{y}$ of the universal cover
 $\pb(X)_{y}$ of $X^{m}$, and the induced fibrations (\ref{eq 10}) are weakly
 equivalent to the fibrations
\begin{equation*}
q: P_{n}\pb(X)_{y} \to P_{n-1}\pb(X)_{y}.
\end{equation*}
It therefore follows from Remark
 \ref{rem 25} that there are homotopy fibre sequences of diagrams
\begin{equation*}
\pb(P_{n}X) \xrightarrow{q_{\ast}} \pb(P_{n-1}X) \xrightarrow{k_{q}} 
P_{n+1}(\pb (P_{n-1}X)/\pb (P_{n}X)) =: Z_{n}X
\end{equation*}
over $E_{\mathbf{\Delta}} H$. The image of $\pb(P_{n}X)$ in
$Z_{n}X$ defines a global base point, and so
Lemma \ref{lem 23} implies that $Z_{n}X$ is weakly equivalent to
$K(H_{n+1}(Z_{n}X),n+1)$ in the $E_{\mathbf{\Delta}} H$-diagram
category. 

There is an isomorphism of groups 
\begin{equation*}
H_{n+1}(Z_{n}X)(\mathbf{m},y) \cong \pi_{n}(X^{m},y)
\end{equation*}
for each $y \in X^{m}$. For more general $x \in H^{m}$ there
is a non-canonical isomorphism
\begin{equation*}
H_{n+1}(Z_{n}X)(\mathbf{m},x) \cong \pi_{n}(X^{m},y)
\end{equation*}
which induced by a morphism $x \to j(y)$ of $H^{m}$.

Taking homotopy colimits preserves homotopy cartesian squares, and we
have proved

\begin{theorem}\label{th 26}
Suppose that $X$ is a cosimplicial Kan complex and suppose that $n
\geq 2$. Then there is a homotopy cartesian diagram
\begin{equation}\label{eq 11}
\xymatrix{
P_{n}X \ar[r] \ar[d]_{q} & B\St(\pi(X)) \ar[d] \\
P_{n-1}X \ar[r]_-{k_{q\ast}} & \hocolim_{H} K(H_{n}(Z_{n}X),n+1)
}
\end{equation}
in cosimplicial spaces,
where $Z_{n}X = P_{n+1}(\pb(P_{n-1}X)/\pb(P_{n}X))$ as a diagram over
the Grothendieck construction $E_{\mathbf{\Delta}} H$ of
the stack completion $H$ of the fundamental groupoid $\pi(X)$.
\end{theorem}

The $k$-invariant $k_{q}: \pb(P_{n-1}X) \to K(H_{n}(Z_{n}X),n+1)$
represents an element of the stack cohomology group 
\begin{equation*}
[\pb(P_{n-1}X),
  K(H_{n}(Z_{n}X),n+1)]
\end{equation*}
where stack cohomology is interpreted to mean abelian group cohomology
for diagrams over the category $E_{\mathbf{\Delta}} H$ --- see
\cite{J39}.
\medskip

The ideas displayed in this section admit substantial
generalization. One could, for example, start with an $I$-diagram $X$
of Kan complexes, and observe that its Postnikov tower is defined over
the $I$-diagram $\pi(X)$ of fundamental groupoids, as well over the
stack completion $H$, which is an injective fibrant model of
$\pi(X)$. Then one shows that the comparison $q_{\ast}: \pb(P_{n}X)
\to \pb(P_{n-1}X)$ of associated diagrams on the Grothendieck
construction $E_{I} H$ has the formal properties that we
saw in the proof of Theorem \ref{th 26}, so that the sequence
\begin{equation*}
\pb(P_{n}X) \xrightarrow{q_{\ast}} \pb(P_{n-1}X) \xrightarrow{k_{q}} 
P_{n+1}(\pb (P_{n-1}X)/\pb (P_{n}X)) =: Z_{n}X
\end{equation*}
is a homotopy fibre sequence of diagrams, and $Z_{n}X$ is a diagram of
Eilenberg-Mac Lane spaces having a global base point.  It follows
that there are homotopy cartesian diagrams of the form (\ref{eq 11})
for all such $I$-diagrams $X$.

\bibliographystyle{plain} 
\bibliography{spt}

\end{document}